\documentclass[11pt]{amsart}
\usepackage{amssymb}
\usepackage{amscd}
\usepackage{bbm}
\usepackage{color}
\usepackage{comment}
\usepackage[dvipsnames]{xcolor}
\usepackage{url}
\usepackage{hyperref}
\usepackage{enumitem}

\newtheorem{theorem}{Theorem}[section]
\newtheorem{lemma}[theorem]{Lemma}
\newtheorem{proposition}[theorem]{Proposition}
\newtheorem{corollary}[theorem]{Corollary}

\theoremstyle{definition}
\newtheorem{definition}[theorem]{Definition}

{\par \hfill \fbox{}}

\newtheorem*{theorem*}{Theorem}

\theoremstyle{definition}

\theoremstyle{thIonascu}

\newtheorem*{thm-Wermer}{Theorem (Wermer, 1952)}
\theoremstyle{thm-Wermer}

\newtheorem*{theorem-Nikolskii}{Theorem (Nikolskii)}
\theoremstyle{theorem-Nikolskii}

\newtheorem*{theorem-FJKP}{Theorem (Foias, Jung, Ko and Pearcy)}
\theoremstyle{theorem-FJKP}

\newtheorem*{theorem-FX}{Theorem (Fang and Xia)}
\theoremstyle{theorem-FX}

\theoremstyle{remark}
\newtheorem{remark}[theorem]{Remark}

\numberwithin{equation}{section}

\newcommand{\norm}[1]{\left\lvert\left\lvert#1\right\rvert\right\rvert}

\DeclareMathOperator*{\Span}{span \;}

\newcommand{\ran}{\textnormal{ran}}
\newcommand{\Impart}{\textnormal{Im}}
\newcommand{\A}{\mathcal{A}}
\newcommand{\Dom}{\textnormal{Dom}}

\newcommand{\pe}[1]{\langle#1\rangle}
\newcommand{\EL}{\mathcal{L}}

\newcommand{\C}{\mathbb{C}}

\newcommand{\N}{\mathbb{N}}
\newcommand{\D}{\mathbb{D}}
\newcommand{\T}{\mathbb{T}}
\newcommand{\PR}{\textnormal{Re}}

\newcommand{\al}{\alpha}

\newcommand{\sumn}{\sum_{n=1}^{\infty}}

\newcommand{\R}{\mathbb{R}}

\newcommand{\Fi}{\varphi}

\newcommand {\conm}[1]{{\{#1\}'}}
\newcommand{\biconm}[1]{{\{#1\}''}}
\newcommand{\X}{\mathcal{X}}

\newcommand{\alg}{{\textnormal{alg}}}
\newcommand{\inte}{\textnormal{int}}

\begin{document}
	
	\title[On commutants of composition operators embedded into $C_0$-semigroups]{On commutants of composition operators \\ embedded into $C_0$-semigroups}

	\author{F. Javier Gonz\'alez-Doña}
	\address{F. Javier González-Doña \newline
		Departamento de Matemáticas, \newline
		Escuela Politécnica Superior, \newline
		Universidad Carlos III de Madrid, \newline
		Avenida de la Universidad 30, 28911 Leganés, Madrid, Spain
	\newline
	and Instituto de Ciencias Matem\'aticas ICMAT (CSIC-UAM-UC3M-UCM),
	\newline Madrid,  Spain}
	\email{fragonza@math.uc3m.es}
	\thanks{The author is partially supported by Plan Nacional  I+D grant no. PID2022-137294NB-I00, Spain,
		the Spanish Ministry of Science and Innovation, through the ``Severo Ochoa Programme for Centres of Excellence in R\&D'' (CEX2019-000904-S) and from the Spanish National Research Council, through the ``Ayuda extraordinaria a Centros de Excelencia Severo Ochoa'' (20205CEX001).}
	
	
	\date{June 2024}
	
	
	\begin{abstract}
Let $C_\Fi$ be a composition operator acting on the Hardy space of the unit disc $H^p$  ($1\leq p < \infty$), which is embedded in a $C_0$-semigroup of composition operators $\mathcal{T}=(C_{\Fi_t})_{t\geq 0}.$ We investigate whether the commutant or the bicommutant of $C_\Fi$, or the commutant of the semigroup $\mathcal{T}$, are isomorphic to subalgebras of continuous functions defined on a connected set. In particular, it allows us to derive results about the existence of non-trivial idempotents (and non-trivial orthogonal projections if $p=2$) lying in such sets. Our methods also provide results concerning the minimality of the commutant and the double commutant property, in the sense that they coincide with the closure in the weak operator topology of the unital algebra generated by the operator. Moreover, some consequences regarding the extended eigenvalues and the strong compactness of such operators are derived. This extends previous results of Lacruz, León-Saavedra, Petrovic and Rodríguez-Piazza, Fernández-Valles and Lacruz and Shapiro on linear fractional composition operators acting on $H^2$.
	\end{abstract}
	\maketitle
	\section{Introduction and Preliminaries}\label{seccion introduccion}
For $1\leq p < \infty,$ we consider the classical Hardy space $H^p,$ which is the space formed by all the holomorphic functions $f:\D\rightarrow \C$ such that the quantity

$$\norm{f}_{H^p} = \left(\sup\limits_{0<r<1} \int_0^{2\pi} |f(re^{i\theta})|^p d\theta\right)^{1/p}$$
is finite. In particular, $H^p$ is a Banach space endowed with the norm $\norm{\cdot}_{H^p}$ for every $1\leq p < \infty.$

\medskip

In general, if $\X$ is a complex Banach space and $\EL(\X)$ is the Banach algebra of linear bounded operators acting on $\X$, we consider the \textit{commutant} and \textit{bicommutant} of an operator $T\in \EL(\X)$, defined as:
\begin{equation*}
\conm{T} = \{A \in \EL(\X): AT=TA \} \qquad \biconm{T} = \bigcap\limits_{A \in \conm{T}} \conm{A},
\end{equation*}
 respectively.
	
	\medskip
	
The study of the commutant and the bicommutant of an operator $T\in \EL(\X)$ provides important consequences regarding the structure of the operator. In this respect, the commutant and bicommutant of composition operators acting on $H^p$ have been widely considered in the recent years. Recall that, given a holomorphic map $\Fi:\D\rightarrow \D$, it defines a composition operator $C_\Fi$ that acts as $$C_\Fi f =f\circ \Fi, \qquad f \in H^p.$$ These operators are well known to be bounded on $H^p$ for all $1\leq p<\infty$ by means of the Littlewood's subordination Theorem.

Our starting point are the works of Lacruz, León-Saavedra, Petrovic and Rodríguez-Piazza \cite{LLPR, LLPR2}, in which they studied the minimality of the commutant and the bicommutant of composition operators $C_\Fi$ acting on $H^2$ induced by linear fractional maps. More specifically, they characterized whether $C_\Fi$ has \textit{minimal commutant}, in the sense that $$\overline{\alg(C_\Fi)}^\sigma = \conm{C_\Fi},$$ where $\alg(C_\Fi) = \{p(C_\Fi) : p \ \textnormal{is a polynomial}\}$ and $\overline{A}^\sigma$ stands for the closure in the weak operator topology of a set $A \subset \EL(\X)$. Likewise, they characterized whether $C_\Fi$ has the \textit{double commutant property}, in the sense that
$$ \overline{\alg(C_\Fi)} = \biconm{C_\Fi}.$$ For more on these properties and (bi)commutants of composition operators, see the references in the aforementioned works.

\medskip

In this work, we are primarily interested in studying the existence of non-trivial idempotents (that is, operators $J\in \EL(H^p)$ such that $J^2=J$ and $J\neq 0, Id_{H^p}$) lying in the commutant or the bicommutant of composition operators, and of non-trivial orthogonal projections, which are self-adjoint idempotents, whenever $p=2$.

\medskip

 Recall that, if such an idempotent  $J$ lies in the commutant (or respectively, in the bicommutant) of $T\in \EL(H^p)$, then the ranges of $J$ and its complement $I-J$, namely $\ran(J)$ $\ran(I-J)$, verify that $H^p = \ran(J)+\ran(I-J)$, and are non-trivial closed invariant subspaces for $T$ (or respectively, hyperinvariant subspaces, which are invariant under every operator $A\in \conm{T}$). If $p=2$ and $J$ is a non-trivial orthogonal projection, these subspaces turn out to be \textit{reducing} for the operator $T$, meaning that $\ran(I-J) = \ran(J)^\perp$ and $H^2=\ran(J)\oplus \ran(J)^\perp.$ 

\medskip

In this sense, the idea of studying the existence of reducing subspaces for composition operators is not novel. For instance, in \cite{Guyker} it is shown that every composition operator $C_\Fi\in \EL(H^2)$ such that $\Fi(0)=0$ has non-trivial reducing subspaces. Moreover,  in \cite{MPS} the lattice of invariant subspaces of composition operators induced by non-automorphic linear fractional parabolic maps acting on $H^2$ is characterized, deriving as a consequence the non-existence of non-trivial reducing subspaces for such operators. For more works in this line for different classes of operators, see \cite{GG3,GG4} and the references therein.

\medskip

The existence of idempotents lying on $\conm{C_\Fi}$ or $\biconm{C_\Fi}$ will be carried out by studying the \textit{functional commutant property} and the \textit{functional bicommutant property} for composition operators which, roughly speaking, ask such sets to be algebraically isomorphic to certain subalgebras of continuous functions. These general concepts, along with a result that connects this notion with the existence of non-trivial idempotents lying in $\conm{C_\Fi}$ or $\biconm{C_\Fi}$, will be introduced in Section \ref{seccion functional commutant}.

\medskip

Regarding composition operators, our methods provide results not only for linear fractional composition operators, but for a quite larger class of composition operators, namely those which can be embedded in a $C_0-$semigroup. Recall that a one parameter family $\mathcal{T}=(T_t)_{t\geq 0}\subset \EL(\X)$ is said to be a \textit{$C_0$}-semigroup if:
\begin{enumerate}
	\item $T_0 = Id_\X.$
	\item $T_{t+s} = T_tT_s$ for all $t,s\geq 0.$
	\item $T_t \rightarrow Id_\X$ in the strong operator topology when $t\rightarrow 0.$
\end{enumerate}
Notice that, given a $C_0-$semigroup, one may also study the existence of non-trivial idempotents in the commutant of the semigroup, that is, the Banach algebra
$$\mathcal{T}' = \bigcap\limits_{t> 0} \conm{T_t}.$$ Observe that, in general, if $T = T_{t_0}$ for some $t_0>0$, $$\biconm{T} \subseteq \mathcal{T}'\subseteq \conm{T}.$$

Now, in order to study $C_0-$semigroups of composition operators, let us introduce the \textit{holomorphic semiflows}. A one parameter family $(\Fi_t)_{t\geq 0}$ of holomorphic functions mapping $\D$ to itself is said to be a \textit{holomorphic semiflow} (or a semigroup of holomorphic functions in $\D$, or simply semigroup in $\D$) if the following conditions hold:
\begin{enumerate}
	\item [(i)] $\Fi_0(z) = z$ for every $z \in \D.$
	\item [(ii)] $\Fi_{t+s} = \Fi_t \circ \Fi_s$ for every $t,s \geq 0.$
	\item [(iii)] If $t\rightarrow s,$ then $\Fi_t \rightarrow \Fi_s$ uniformly on compact subsets of $\D$ for $t,s\geq 0$.
\end{enumerate}

It is well known, by the seminar work of Berkson and Porta \cite{BP}, that a one parameter family of composition operators $(C_{\Fi_t})_{t\geq 0}$ is a $C_0-$semigroup if and only if  $(\Fi_t)_{t\geq 0}$ is a holomorphic semiflow. It is worth remarking that every linear fractional map which is not loxodromic can be embedded in a holomorphic semiflow, that is, there exists a holomorphic flow $(\Fi_t)_{t\geq 0}$ such that $\Fi = \Fi_1$. In particular, its induced composition operator can be embedded in a $C_0-$semigroup. It is an open problem to decide whether a given holomorphic map $\Fi : \D\rightarrow \D$ can be embedded in a holomorphic semiflow (see, for instance, \cite[Section 4.3 and Chapter V]{Shoikhet}). 

\medskip

The holomorphic semiflows can be classified in terms of the configuration of common fixed points of the holomorphic maps that comprise it. A holomorphic semiflow $(\Fi_t)_{t\geq 0}$ is \textit{elliptic} if there exists $z_0 \in \D$ such that $\Fi_{t_0}(z_0) = z_0$ for some $t_0>0.$ If the holomorphic semiflow is not elliptic, then there exists a unique point $\tau \in \T$ such that $\Fi_t$ converges uniformly on compact subsets to the constant function $\tau$ when $t\rightarrow \infty$. This point $\tau$ is called the \textit{Denjoy-Wolff point} of $(\Fi_t)_{t\geq 0}.$ It is known that the functions $\Fi_t'$ has non-tangential limits $e^{t\al}$ at the Denjoy-Wolff point $\tau,$ with $\al \leq 0$. Thus, the holomorphic semiflow is said to be \textit{hyperbolic} if $\al < 0$ and \textit{parabolic} if $\al = 0.$

\medskip

Our results regarding the (bi)commutant of composition operators which can be embedded in $C_0$-semigroups will depend on the nature of the holomorphic semiflow, regarding the classification discussed above. In order to present our main results, we need to introduce some geometric aspects of the semiflows.

\medskip

For every non-elliptic holomorphic semiflow $(\Fi_t)_{t\geq 0}$ there exists a univalent function $h:\D\rightarrow \C$  such that 
\begin{equation}\label{koenigs function}
	\Fi_t(z) = h^{-1}(h(z)+t) \qquad  z \in \D, \ t\geq 0.
\end{equation}
Such a function $h$ is called the \textit{Koenigs function} of $(\Fi_t)_{t\geq 0}$. The set $\Omega := h(\D)$ is said to be the \textit{Koenigs domain} of $(\Fi_t)_{t\geq 0}.$ Such a domain is \textit{convex at infinity}, that is, if $z \in \Omega$ and $t\geq 0$, then $z+t \in \Omega$. Moreover, if there exists another univalent map $g:\D\rightarrow \C$ satisfying $g(\D)=h(\D)$ and \eqref{koenigs function}, then $g(z) = h(z)+a$ for some $a \in \C$ \cite[Proposition 9.3.10]{BCD}.

\medskip

The geometry of the Koenigs domain encodes valuable information for the associated holomorphic semiflow. For instance, a non-elliptic holomorphic semiflow is hyperbolic if and only if its Koenigs domain is contained in a horizontal strip \cite{CD}. In this sense, the geometry of the Koenigs domain will also play a fundamental role in our results.

\medskip

Now, we can state our main result regarding hyperbolic semiflows, which are studied in Section \ref{seccion hiperbolicos}. It reads as follows:

\medskip

\noindent \textbf{Theorem 1.} (Theorems \ref{teorema hiperbolico} and \ref{teorema conmutante no abeliano hiperbolico}). \textit{Let $\Fi : \D \rightarrow\D$ be a holomorphic function such  that there exists a hyperbolic  flow $(\Fi_t)_{t\geq 0}$ with $\Fi = \Fi_1.$ Then, for every $1\leq p < \infty,$
	\begin{enumerate}
		\item [(i)] Let $\Omega$ be the associated Koenigs domain of $(\Fi_t)_{t\geq 0}$  and assume that $\inte(\overline{\Omega}) = \Omega$ and that $\C\setminus \overline{\Omega}$ has at most two connected components.  Then, $C_\Fi \in \EL(H^p)$ has the functional bicommutant property. In particular, there exist no non-trivial idempotents lying in $\biconm{C_\Fi}.$
		\item [(ii)] The commutant of $C_\Fi \in \EL(H^p)$ is not abelian, so $C_\Fi$ does not have the functional commutant property or minimal commutant.
\end{enumerate}}
 We also show that, under the hypotheses on ($i$), if $\mathcal{T} = (C_{\Fi_t})_{t\geq 0}\subset \EL(H^p)$, then $\mathcal{T}'$ does not contain any non-trivial idempotents (Corollary \ref{corolario conmutante semigrupo hiperbolico}). Moreover, we also prove a result providing sufficient conditions for $C_\Fi$ to have the double commutant property (Theorem \ref{teorema doble conmutante hiperbólico}). 

\medskip

Parabolic semiflows are considered in Section \ref{seccion parabolicos}. In order to state our main results in this setting, we need to distinguish two cases.  A parabolic semiflow $(\Fi_t)_{t\geq 0}$ is said to be of \textit{positive hyperbolic step} if for some $z \in \D$ the quantity $$\lim\limits_{t\rightarrow +\infty} k_\D(\Fi_t(z),\Fi_{t+1}(z))$$ is positive, where $k_\D$ denotes the hyperbolic distance in $\D$. If the previous situation does not hold, $(\Fi_t)_{t\geq 0}$ is of \textit{zero hyperbolic step}.

\medskip

Our main result regarding the bicommutant is the following:

\medskip

\noindent \textbf{Theorem 2.} (Theorem \ref{teorema parabolico}). \textit{Let $\Fi : \D \rightarrow\D$ be a holomorphic function such  that there exists a parabolic semiflow $(\Fi_t)_{t\geq 0}$ with $\Fi = \Fi_1.$ Let $\Omega$ be its Koenigs domain. Assume that either
	\begin{enumerate}
		\item [(i)] $(\Fi_t)_{t\geq 0}$ is of positive hyperbolic step, $\inte(\overline{\Omega})=\Omega$ and $\C\setminus\overline{\Omega}$ is connected; or
		\item [(ii)]  $(\Fi_t)_{t\geq 0}$ is of zero hyperbolic step, $\Omega$ is contained in a half-plane and $\inte(\overline{\Omega})= \Omega.$
	\end{enumerate}
	Then,  $C_\Fi \in \EL(H^p)$ has the functional bicommutant property for $1\leq p<\infty$. In particular, there exist no non-trivial idempotents lying in $\biconm{C_\Fi}.$}

\medskip

As in the hyperbolic case, the previous result can also be stated for the commutant of $\mathcal{T} = (C_{\Fi_t})_{t\geq 0}$ (Corollary \ref{corolario conmutante semigrupo parabolico}). On the other hand, the situation for $\conm{C_\Fi}$ can be quite different, regarding the geometric picture of the Koenigs domain:

\medskip

\noindent \textbf{Theorem 3.} (Theorems \ref{functional commutant parabolico} and \ref{conmutante abeliano parabolico}). \textit{Let $\Fi : \D \rightarrow\D$ be a holomorphic function such  that there exists a parabolic semiflow $(\Fi_t)_{t\geq 0}$ with $\Fi = \Fi_1,$ and let $\Omega$ be its Koenigs domain.
	\begin{enumerate}
\item [(i)] If $\Omega$ is contained in a half plane, $\inte(\overline{\Omega})=\Omega$ and $\sumn 1-|\Fi_n(0)| = \infty,$ then $C_\Fi \in \EL(H^p)$ has the functional commutant property for $1\leq p<\infty$. In particular, there exist no non-trivial idempotents lying in $\conm{C_\Fi}.$
\item [(ii)] If $\Omega$ is contained in a horizontal half-plane of the form $\{w\in\C: \Impart(w)> -\eta\}, \eta >0$, $\conm{C_\Fi}$ is not abelian. In particular, $C_\Fi$ does not have the functional commutant property or minimal commutant.
	\end{enumerate} }

\medskip

As a particular instance, $(i)$ can be applied to deduce that, if $\Fi$ is a non-automorphic parabolic linear fractional map, then $C_\Fi\in \EL(H^p)$ does not commute with any non-trivial idempotent (Corollary \ref{corolario parabolico no-automorfismo}), which extends the aforementioned result of Montes-Rodríguez, Ponce-Escudero and Shkarin.

\medskip

For the elliptic semiflow case, which is considered in Section \ref{seccion elipticos}, we study the more general situation of holomorphic maps $\Fi:\D\rightarrow \D$ fixing a point in $\D$. In this respect, our main result is the following: 

\medskip

\noindent \textbf{Theorem 4.} (Theorem \ref{teorema eliptico}). \textit{Let $\Fi :\D\rightarrow \D$ be a holomorphic function with a fixed point in $\D$ and consider $C_\Fi \in \EL(H^p)$, $1< p < \infty.$ Then, there exist non-trivial idempotents lying in $\biconm{C_\Fi}.$ As a matter of fact, $C_\Fi$ does not have the functional commutant or bicommutant property.}

\medskip

Let us point out that our methods rely on some properties of the infinitesimal generator of the $C_0-$semigroup $(C_{\Fi_t})_{t\geq 0},$ which is a linear closed densely-defined unbounded operator given by 
\begin{equation}\label{expresion generador infinitesimal}
\Delta f = \lim\limits_{t\to 0} \frac{C_{\Fi_t}f-f}{t} =Gf' \qquad (f\in \Dom(\Delta)),
\end{equation}
 where $G$ is a unique holomorphic function $G:\D\rightarrow \C$ satisfying $$\frac{\partial\Fi_t(z)}{\partial t} = G(\Fi_t(z)), \qquad (t\geq 0, z \in \D).$$

 A remarkable property is that the point spectrum of $\Delta$ can be related with the point spectrum of the operators of the semigroup. More specifically, the following equality
\begin{equation}\label{spectral mapping theorem point spectrum}
	\sigma_p(C_{\Fi_t}\mid_{H^p})  = \{ e^{t\lambda}: \lambda\in \sigma_p(\Delta\mid_{H^p}) \} \qquad (t>0)
\end{equation}
follows from \cite[Chapter IV, Theorem 3.7]{EN} and the fact that $C_{\Fi_t}$ is injective for all $t>0.$  It is also notable that $\lambda \in \sigma_p(\Delta)$ if and only if the function $e^{\lambda h}$ belongs to $H^p$, where $h$ is the Koenigs function of $(\Fi_t)_{t\geq 0}.$ In such a case, the eigenspace $\ker(\Delta-\lambda I)$ is one-dimensional and it is spanned by $e^{\lambda h}$ \cite{Siskakis}. 

\medskip

The characterization of the point spectrum of the infinitesimal generator for non-elliptic semigroups, which was carried out by Betsakos in \cite{Betsakos}, will be essential to state our results, along with the results of Bracci, Gallardo-Gutiérrez and Yakubovich \cite{BGY}, who established results about the density of the eigenfunctions of the infinitesimal generator. Since these results depend on the kind of holomorphic semiflows we are considering, we will present these statements separately for hyperbolic semiflows and parabolic semiflows.

\medskip

It is also worth mentioning that our methods will also provide characterizations of the extended eigenvalues and the strong compactness of such operators (see Subsections \ref{subseccion 1} and \ref{subseccion 2}). These two notions have been studied for linear fractional composition operators in $H^2$ in \cite{FL, LLPR3, Shapiro}. For more on the subjects, see the references in these works.

\section{The functional (bi)commutant property}\label{seccion functional commutant}
In this section, we introduce the functional commutant property and the functional bicommutant property, and study their connection with the existence of non-trivial idempotents lying in such sets. In order to introduce these concepts, given a subset $G \subset \C$, we denote by $\mathcal{C}(G)$ the algebra consisting of all continuous functions $f:G \rightarrow \C$, equipped with the pointwise sum and multiplication of functions. Notice that we are not endowing $\mathcal{C}(G)$ with any topology, since we are only interested in the algebraic properties of the set. 

\begin{definition}\label{definicion functional commutant}
	We say that an operator $T\in \EL(X)$ has the \textit{functional commutant property} (respectively, \textit{functional bicommutant property}) if there exist a connected subset $G \subset \C$, a subalgebra $\A\subseteq \mathcal{C}(G)$ and an algebraic isomorphism $\Phi : \A \rightarrow \conm{T}$ (respectively, $\Phi: \A \rightarrow \biconm{T}).$
\end{definition}
Although this definition is, at least to the author's knowledge, novel in the literature, the concept of the functional commutant (or bicommutant) property for an operator has been implicitly studied and already encompasses several examples.

\medskip

For instance, in \cite{SW} the authors showed that under quite reasonable hypothesis on a Hilbert space of holomorphic functions, $\conm{M_z} = H^\infty(\D)$, and it is easy to construct an algebra isomorphism $\Phi : H^\infty(\D)\rightarrow \conm{M_z}$, so the operator $M_z$ has the  functional commutant property. This can also be used to prove that the classical Cesàro operator has the functional commutant property. Similar ideas are also managed by Sarason \cite{Sarason} when describing the commutants of the Volterra operator on $H^2$ and of some truncated Toeplitz operators. On the other hand, it was proved in \cite[Theorems 3.4 and 5.6]{LLPR2} that the composition operator $C_\Fi$ has the functional bicommutant property whenever $\Fi$ is a  hyperbolic linear fractional map with no fixed points in $\D$.

\medskip

Now, let us state the aforementioned relation between  the functional commutant (or bicommutant) property  and the existence of idempotents lying in such algebras. 

\begin{theorem}\label{main result}
	Let $T\in \EL(\X)$ and assume it has the functional commutant property (respectively, the functional bicommutant property). Then, there exist no non-trivial idempotents lying in $\conm{T}$ (respectively, $\biconm{T}$). 
\end{theorem}
\begin{proof}
	Let us prove the result for the commutant, the proof for the bicommutant is analogous. By contradiction, assume that there exists a non-trivial idempotent $J\in \conm{T}.$ Thus, since $T$ has the functional commutant property, there exist a connected subset $G \subset \C$, a subalgebra $\A\subseteq \mathcal{C}(G)$ and an algebra isomorphism $\Phi : \A \rightarrow\conm{C_\Fi}.$ In particular, there exists $f\in \A$ such that $\Phi(f) = J.$ Since $\Phi$ is an algebra isomorphism, it is straightforward to check that $f$ is also a non-trivial idempotent, that is, $f^2 = f$ and $f$ is not the zero or the constant function $f\equiv 1.$ But $f$ is a continuous function defined on a connected set $G,$ so $f^2 = f$ implies that $f\equiv 0$ of $f\equiv 1,$ which yields the desired contradiction.
\end{proof}
To finish this section, let us state a simple lemma that will be important in what follows:

\begin{lemma}\label{lema conmutante abeliano}
	Let $T\in \EL(X)$ and suppose that either
	\begin{enumerate}
		\item [(i)] $T$ has the functional commutant property; or
		\item [(ii)]  $T$ has minimal commutant.
	\end{enumerate}
 Then, $\conm{T}$ is abelian and $\conm{T}= \biconm{T}.$
\end{lemma}
\section{Hyperbolic semiflows}\label{seccion hiperbolicos}

We begin by presenting a slight reformulation of the description of the point spectrum of the infinitesimal generators associated to hyperbolic semigroups, which was carried out by Betsakos:

\begin{theorem}\label{teorema betsakos} \cite{Betsakos}
	Let $(\Fi_t)_{t\geq 0}$ be a hyperbolic semiflow, let  $\Omega$ be the associated Koenigs domain and $\gamma$ be the width of the smallest horizonal strip containing $\Omega.$ Let $(C_{\Fi_t})_{t\geq 0}$ be the corresponding $C_0$-semigroup of composition operators acting on $H^p$ for $1\leq p < \infty,$ and let $\Delta$ be its infinitesimal generator.
	\begin{enumerate}
		\item [(i)] If $\Omega$ does not contain any horizontal strip, then
		$$ \{ \lambda \in \C: - \infty < \PR(\lambda)< \frac{\pi}{p\gamma} \} \subseteq \sigma_p(\Delta\mid_{H^p})\subseteq \{ \lambda \in \C: - \infty < \PR(\lambda) \leq \frac{\pi}{p\gamma} \}.$$
		\item If $\Omega$ contains a horizontal strip, there exists a horizontal strip contained in $\Omega$ with width 
		$$\beta_{\textnormal{max}} = \max \{ \beta>0 : \Omega \ \textnormal{contains a horizontal strip of width } \beta \}.$$
		Moreover,
		$$ \{ \lambda \in \C: - \frac{\pi}{p\beta_{\textnormal{max}}} <\PR(\lambda)<\frac{\pi}{p\gamma} \} \subseteq \sigma_p(\Delta\mid_{H^p})\subseteq \{ \lambda \in \C: - \frac{\pi}{p\beta_{\textnormal{max}}} <\PR(\lambda)\leq\frac{\pi}{p\gamma} \}.$$
	\end{enumerate}
\end{theorem}
Now, we state a result about the density of the linear manifolds spanned by the eigenvectors of the infinitesimal generator $\Delta$. This result is just a small modification of \cite[Theorem 1.1]{BGY}, in order to consider only the interior of the point spectrum to obtain the density of certain linear manifolds spanned by eigenvectors:
\begin{proposition}\cite{BGY}\label{proposicion densidad}
	Let $(\Fi_t)_{t\geq 0}$ be a hyperbolic semiflow, let $\Omega$ be its Koenigs domain and assume that $\inte(\overline{\Omega}) = \Omega$ and that $\C\setminus \overline{\Omega}$ has at most two connected components. Then, if $h$ is the Koenigs function associated to $(\Fi_t)_{t\geq 0}$, the linear manifold $$\Span \{e^{\lambda h} : \lambda \in \inte(\sigma_p(\Delta\mid_{H^p})) \}$$ is dense in $H^p$ for $1\leq p < \infty.$
\end{proposition}
\begin{proof}
 It is equivalent to show that $\Span\{e^{\lambda z} : \lambda \in \inte(\sigma_p(\Delta\mid_{H^p})) \}$ is dense on the Hardy space $H^p(\Omega)$ associated to the domain $\Omega$, that is, the Banach space of all the holomorphic functions $f:\Omega \rightarrow \C$ such that there exists a harmonic function $u: \Omega \rightarrow \C$ such that $|f(z)|^p\leq u(z)$ for every $z \in \Omega$ (see \cite{Zhu} for more on $H^p$ spaces of general domains). 

Now, by the convexity at infinity of $\Omega$ it is clear that $$\{ e^{\lambda z} \in H^\infty(\Omega) : \lambda \in \C \} \subseteq \{e^{\lambda z} \in H^p(\Omega) : \lambda \in \sigma_p(\Delta\mid_{H^p}) \}.$$ As a consequence of our hypotheses on $\Omega$ and \cite[Theorem 1.1]{BGY} we deduce that $\Span\{ e^{\lambda z} \in H^\infty(\Omega) : \lambda \in \C \}$ is weak-star dense in $H^\infty(\Omega)$, so it is weak dense on $H^p(\Omega)$ for $1<p<\infty.$ Thus, Mazur's Lemma (see \cite[Corollary 3, Chapter 2]{Diestel}) ensures the density in $H^p(\Omega)$ for $1<p< \infty$, and the density in $H^1(\Omega)$ follows from the continuity of the injection $H^p(\Omega)\hookrightarrow H^1(\Omega).$
\end{proof}
It will also be important to handle a slightly stronger property concerning the density of eigenfunctions of operators. An operator $T\in\EL(\X)$ has \textit{rich point spectrum} if $\inte(\sigma_p(T))\neq \emptyset$ and, for every open disc $D\subset \sigma_p(T)$, the linear manifold $$\Span \bigcup\limits_{z \in D} \ker(T-zI)$$ is dense in $\mathcal{X}.$ This concept was first introduced in \cite{LLPZ} in order to study the extended eigenvalues of Cesàro operators. 

Before stating the aforementioned result, let us recall that $(H^p)^*$ can be identified with $H^q$ for $1<p<\infty$, where $\frac{1}{p}+\frac{1}{q} =1$, while $(H^1)^*$ may be identified with $BMOA$ \cite{Zhu}. In particular, for each functional $\delta \in (H^p)^*$ there exists a unique $g \in H^q$ (or $BMOA$, respectively) such that $$\delta(f) = \int_{0}^{2\pi} f(e^{i\theta})\overline{g(e^{i\theta})}d\theta \qquad (f\in H^p).$$

\begin{lemma}\label{lema rich point spectrum}
	Let $\Fi : \D \rightarrow\D$ be a holomorphic function such  that there exists a hyperbolic holomorphic semiflow $(\Fi_t)_{t\geq 0}$ with $\Fi = \Fi_1.$ Let $\Omega$ be its Koenigs domain and assume that $\inte(\overline{\Omega}) = \Omega$ and that $\C\setminus \overline{\Omega}$ has at most two connected components. Then, $C_\Fi$ has rich point spectrum.
\end{lemma}
\begin{proof}
	By Theorem \ref{teorema betsakos} and \eqref{spectral mapping theorem point spectrum} it is clear that $\inte(\sigma_p(C_\Fi))\neq \emptyset.$ Let $D\subset \sigma_p(C_\Fi)$ be an open disc. We are showing that $$\Span \{ e^{\lambda h} : e^\lambda \in D  \}$$ is dense in $H^p$, which will yield the result. For such a task, observe that the map $$\lambda \in \inte(\sigma_p(\Delta\mid_{H^p})) \mapsto e^{\lambda h} \in H^p$$ is a vector-valued holomorphic function, since for every $g \in H^q$ with $\frac{1}{p}+\frac{1}{q}= 1$ if $p>1$ of $g\in BMOA$ if $p=1$ the map $$\lambda \in \inte(\sigma_p(\Delta\mid_{H^p})) \mapsto \pe{e^{\lambda h},g} = \int_0^{2\pi} e^{\lambda h(e^{i\theta})}\overline{g(e^{i\theta})}d\theta$$ is holomorphic by means of a standard application of Morera's Theorem. Now, assume that $1<p<\infty$ and let $g\in H^q$ (again, the case $p=1$ is equivalent considering $g\in BMOA$) satisfying $\pe{e^{\lambda h},g} = 0$ for every $e^\lambda \in D.$ By the Identity Theorem, it follows then that $\pe{e^{\lambda h},g} = 0$ for every $\lambda \in \inte(\sigma_p(\Delta\mid_{H^p})),$ and Proposition \ref{proposicion densidad} yields that $g=0.$ This shows the density of $\Span \{ e^{\lambda h} : e^\lambda \in D  \}$, and the proof is done.
\end{proof}

\subsection{The functional (bi)commutant property}
Now, we can state the main result for hyperbolic semiflows. The proof includes various ideas from \cite[Section 3]{LLPR2}:
\begin{theorem}\label{teorema hiperbolico}
	Let $\Fi : \D \rightarrow\D$ be a holomorphic function such  that there exists a hyperbolic  semiflow $(\Fi_t)_{t\geq 0}$ with $\Fi = \Fi_1.$ Let $\Omega$ be its Koenigs domain  and assume that $\inte(\overline{\Omega}) = \Omega$ and that $\C\setminus \overline{\Omega}$ has at most two connected components.  Then, $C_\Fi \in \EL(H^p)$ has the functional bicommutant property for every $1\leq p < \infty.$ In particular, there exist no non-trivial idempotents lying in $\biconm{C_\Fi}.$
\end{theorem}
\begin{proof}
	Our aim is to construct a connected open subset $G\subset \C$, a subalgebra $\A \subset \mathcal{C}(G)$ and an algebra isomorphism $\Phi : \A \rightarrow \biconm{C_\Fi}.$ Let $G = \inte(\sigma(\Delta\mid_{H^p})),$ which by Theorem \ref{teorema betsakos} may be a vertical half-plane or a vertical strip. Observe that both sets are connected open sets. We will work with the two situations indifferently, since the proof works the same for both cases.

	 Let $X \in \biconm{C_\Fi}$ and consider $\lambda \in G.$ Observe that $e^{\lambda h} \in H^p$, and in particular $C_{\Fi_t}e^{\lambda h} = e^{\lambda t}e^{\lambda h}$ for every $t> 0$ by \eqref{spectral mapping theorem point spectrum}. Now, 
	 \begin{equation}\label{igualdad biconmutante}
Xe^{\lambda h} = e^{-\lambda t}XC_{\Fi_t}e^{\lambda h} = e^{-\lambda t} C_{\Fi_t} (Xe^{\lambda h}).
	 \end{equation}
	  That is, $Xe^{\lambda h}$ is a common eigenvector for the semigroup $(C_{\Fi_t})_{t\geq 0}$, so there exists a map $\nu_X : G \rightarrow \C$ such that $Xe^{\lambda h} = \nu_X(\lambda)e^{\lambda h}.$ We claim that $\nu_X$ is, indeed, a holomorphic function. For such a task, recall that the map $\lambda \in G \mapsto e^{\lambda h} \in H^p$ is a vector-valued holomorphic function, as it was shown in the proof of Lemma \ref{lema rich point spectrum}. Now, consider $\mathbbm{1} \in H^p$ to be the constant function $\mathbbm{1}(z) \equiv 1$. Recall that $\pe{f,\mathbbm{1}}=f(0)$ for every $f \in H^p.$ Thus,  $$\pe{Xe^{\lambda h},\mathbbm{1}} = \nu_X(\lambda)\pe{e^{\lambda h},\mathbbm{1}} = \nu_X(\lambda)e^{\lambda h(0)},$$ so $\nu_X(\lambda) = e^{-\lambda h(0)}\pe{Xe^{\lambda h},\mathbbm{1}}$ for every $\lambda \in G$. Hence, $\nu_X$ is a product of holomorphic functions and is holomorphic as well, as claimed. 
		
		Now, observe that $\nu_X(\lambda) \in \sigma_p(X)$ for every $\lambda \in G$, so $|\nu_X(\lambda) |\leq \norm{X}$ for every $\lambda \in G$. That is, $\nu_X \in H^\infty(G).$ Then, we can define the map $\Psi : X \in \biconm{C_\Fi} \mapsto \nu_X \in H^\infty(G),$ noticing that 
		\begin{equation}\label{igualdad clave}
Xe^{\lambda h}= \Psi(X)(\lambda) e^{\lambda h} \qquad (\lambda \in G).
		\end{equation}
		 Previous equality \eqref{igualdad clave} guarantees the uniqueness of the definition, so $\Psi$ is well defined. Let us prove that $\Psi$ is an injective algebra homomorphism. The linearity follows immediately as a consequence of \eqref{igualdad clave}. To see that $\Psi$ is multiplicative, let $X,Y \in \biconm{C_\Fi}$ and observe that for every $\lambda \in G$
		 $$\Psi(XY)(\lambda)e^{\lambda h} = XYe^{\lambda h} = X (\Psi(Y)(\lambda) e^{\lambda h}) = \Psi(X)(\lambda)\Psi(Y)(\lambda)e^{\lambda h},$$ so $\Psi(XY) = \Psi(X)\Psi(Y)$, as we wanted. To prove the injectivity of $\Psi$, if $\Psi(X) = 0$, then $Xe^{\lambda h} = 0$ for every $\lambda \in G$ by \eqref{igualdad clave}. At this point, it is enough to apply Proposition \ref{proposicion densidad} to deduce that $X=0$, so $\Psi$ is injective.
		 
		 Thus, $\Psi$ is an injective algebra homomorphism and  $\Psi : \biconm{C_\Fi} \rightarrow \ran(\Psi)$ is an algebra isomorphism. Thus, we may consider $\Phi  := \Psi^{-1}:\ran(\Psi)\rightarrow \biconm{C_\Fi}$, which is again an algebra isomorphism. By construction, it is an immediate fact that $\ran(\Psi)$ is a subalgebra of $\mathcal{C}(G)$, since it contains $H^\infty(G)$ as a subalgebra. As a consequence, we have shown that $C_{\Fi}$ has the functional bicommmutant property, and by Theorem \ref{main result} it has no non-trivial idempotents lying in $\biconm{C_\Fi}.$
\end{proof}
A careful reading of the proof shows that the fact $X\in \biconm{C_\Fi}$ has been only used to derive that $X$ commutes with $C_{\Fi_t}$ for all $t>0.$ Hence, it is straightforward to deduce the following consequence:

\begin{corollary}\label{corolario conmutante semigrupo hiperbolico}
	Let $(\Fi_t)_{t\geq0}$ be a hyperbolic flow such that its Koenigs domain $\Omega$ verifies that $\inte(\overline{\Omega}) = \Omega$ and that $\C\setminus \overline{\Omega}$ has at most two connected components, and denote $\mathcal{T}= (C_{\Fi_t})_{t\geq 0}$ it associated $C_0-$semigroup acting on $H^p$ for $1\leq p <\infty$. Then, $\mathcal{T}'$ is abelian and does not contain any non-trivial idempotent.
\end{corollary}
At this point, it is natural to ask whether every composition operator satisfying the hypotheses of Theorem \ref{teorema hiperbolico} has not only the functional bicommutant property, but also the functional commutant property. In this sense, the following result holds:
\begin{theorem}\label{teorema conmutante no abeliano hiperbolico}
	Let $\Fi : \D \rightarrow\D$ be a holomorphic function such  that there exists a hyperbolic  semiflow $(\Fi_t)_{t\geq 0}$ such that $\Fi = \Fi_1.$ Then, $\conm{C_\Fi}$ is not abelian. In particular, $C_\Fi$ does not have either the functional commutant property or minimal commutant.
\end{theorem}
\begin{proof}
	Let $\Omega$ and $h$ be the Koenigs domain and the Koenigs function associated to $(\Fi_t)_{t\geq 0}$, respectively. Since it is a hyperbolic semiflow, $\Omega$ is contained in a horizontal strip. Then, the functions $e^{ixh}$ belong to $H^\infty$ (they are analytic and bounded on $\D$) for every $x \in \R.$ Moreover, by Theorem \ref{teorema betsakos} and \eqref{spectral mapping theorem point spectrum} it follows that $C_{\Fi_t}e^{ixh} = e^{itx} e^{ixh}$ for every $x,t \in \R.$ Now, for each $x\in \R$, consider the multiplication operator $M_{e^{ix h}},$ which is bounded on $H^p$. Then, for every $f \in H^p$ and $t,x \in \R$
	$$C_{\Fi_t}M_{e^{ixh}}f = C_{\Fi_t}(e^{ ixh})\cdot C_{\Fi_t}f = e^{itx} e^{ixh} C_{\Fi_t}f = e^{itx} M_{e^{ ix h}}C_{\Fi_t}f.$$
	That is, 
	\begin{equation}\label{eigenoperator}
		C_{\Fi_t}M_{e^{ixh}} = e^{itx} M_{e^{ ix h}}C_{\Fi_t}f
	\end{equation}
	At this point, it is easy to consider $x_0 = 2\pi$ to obtain a multiplication operator $M_{e^{i x_0 h}}$ lying in the commutant of $C_{\Fi_1}$ and $t_0 \in \R$ that $C_{\Fi_{t_0}}$ does not commute with $M_{e^{i x_0 h}}$ by means of \eqref{eigenoperator}. This shows that $\conm{C_\Fi}$ is not abelian, and Lemma \ref{lema conmutante abeliano} completes the proof.
\end{proof}
\subsection{The double commutant property}
We can exploit the construction of the proof of Theorem  \ref{teorema hiperbolico} to deduce the double commutant property for a subclass of the composition operators considered:

\begin{theorem}\label{teorema doble conmutante hiperbólico}
	Let $\Fi : \D \rightarrow\D$ be a holomorphic function such  that there exists a hyperbolic holomorphic semiflow $(\Fi_t)_{t\geq 0}$ such that $\Fi = \Fi_1.$ Let $\Omega$ be its Koenigs domain  and assume that $\inte(\overline{\Omega}) = \Omega$ and that $\C\setminus \overline{\Omega}$ has at most two connected components.  With the notation of Theorem \ref{teorema betsakos}, assume further that $\Omega$ does not contain any horizontal strip and that there exists $C>0$ such that 
	\begin{equation}\label{hipotesis norma polinomios}
\norm{p(C_\Fi)} \leq C \cdot \sup\{ |p(z)| : |z| \leq e^{\frac{\pi}{p\gamma}} \}
	\end{equation} for every polynomial $p \in \C[z]$. Then, $C_\Fi \in \EL(H^p)$ has the double commutant property for every $1<p<\infty.$
\end{theorem}
Before proceeding with the proof, let us take a closer look at the inequality \eqref{hipotesis norma polinomios}:

\begin{remark}
Consider the Hilbertian Hardy space $H^2$. It is known, by the  von Neumann's inequality, that for every contraction $T\in \EL(H^2)$ (that is, $\norm{T}\leq 1$), the following inequality
	$$\norm{p(T)} \leq \sup\{ |p(z)| : |z|\leq 1 \}$$ holds for every complex polynomial $p\in\C[z]$. In general, if $T$ is any non-zero operator, one may consider $\frac{T}{\norm{T}}$ and apply the previous inequality to obtain 
	$$\norm{p(T)} \leq \sup\{ |p(z)| : |z|\leq \norm{T} \}.$$
	In addition, it is known by \cite[Corollary 3.2]{Siskakis} and \cite[Theorem 2.1]{CD} that $r(C_\Fi\mid_{H^p}) = e^{\frac{\pi}{p\gamma}}$, where $r(C_\Fi\mid_{H^p})$ denotes the spectral radius of a composition operator $C_\Fi$ acting on $H^p$. So, in the Hilbert space case, \eqref{hipotesis norma polinomios} turns into $\norm{C_\Fi\mid_{H^2}} = r(C_\Fi\mid_{H^2}).$ For instance, it holds if $$\Fi_t(z) = e^{-t}z+1-e^{-t} \qquad (z\in \D,t \geq 0)$$ is the hyperbolic semiflow of non-automorphic linear fractional maps (\cite[Theorem 3]{Cowen}), so we recover the result in \cite{LLPR2}.

	On the other hand, if $p\neq 2,$ a von-Neumann type-inequality does not hold in general. In this setting, in order to obtain composition operators $C_\Fi$ verifying \eqref{hipotesis norma polinomios}, it is enough to ask the operator $C_\Fi/\norm{C_\Fi}$ to be \textit{polynomially bounded} (see, for instance, \cite[Definition 2.3.2]{ChP}), together with the equality $\norm{C_\Fi}_{H^p} = r(C_\Fi \mid_{H^p}).$ Whether the operator $C_\Fi/\norm{C_\Fi}$ is polynomially bounded is an open problem, at least to the author's knowledge.
\end{remark}

	Now, we present the proof of Theorem \ref{teorema doble conmutante hiperbólico}.
	
	\medskip
	
\noindent \textit{Proof of Theorem \ref{teorema doble conmutante hiperbólico}}. Again, this proof is based on various ideas from \cite[Section 3]{LLPR2}. Consider the map $\Psi : \biconm{C_\Fi} \rightarrow H^\infty(G)$ constructed in the proof of Theorem \ref{teorema hiperbolico} and recall that, since we are assuming that $\Omega$ does not contain any horizontal strip, then $G = \{ \lambda \in \C: - \infty < \PR(\lambda)< \frac{\pi}{p\gamma} \}.$ Now, if $r=e^{\frac{\pi}{p\gamma}},$ then the exponential function maps $G$ onto $D(0,r)\setminus \{0\}.$

 Let $X\in \biconm{C_\Fi}.$ We are showing that $\Psi(X)(\lambda) = \Psi(X)(\lambda+2\pi i k)$ for every $\lambda \in G$ and $k \in \mathbb{Z}$, which, together with the Factorization Theorem \cite[p. 286]{Remmert}, will yield that there exists $\xi_X$ holomorphic on $D(0,r)\setminus\{0\}$ such that $\Psi(X)(\lambda) = \xi(e^\lambda)$ for every $\lambda \in G.$
 We have
 $$\Psi(X)(\lambda + 2\pi i k) e^{(\lambda+2\pi ik)h} = X e^{(\lambda+2\pi ik)h} = Xe^{ (2\pi i k)h }e^{\lambda h}.$$
 At this point, observe that $e^{(2\pi i k)h} \in H^\infty$, so the multiplication operator $M_{e^{(2\pi i k)h}}$ is bounded. Indeed, it is straightforward to check that it commutes with $C_\Fi,$ so it also commutes with $X$. Thus, we have
 $$ Xe^{ (2\pi i k)h }e^{\lambda h} = XM_{e^{(2\pi i k)h}} e^{\lambda h} = M_{e^{(2\pi i k)h}}Xe^{\lambda h} =  e^{(2\pi i k)h}\Psi(X)(\lambda)e^{\lambda h} = \Psi(X)(\lambda) e^{(\lambda+2\pi ik)h},$$ so $\Psi(X)(\lambda) = \Psi(X)(\lambda+2\pi i k)$, as we wanted. Thus, $\Psi(X)(\lambda) = \xi(e^{\lambda})$ for every $\lambda \in G$ for some holomorphic function $\xi_X$ on $D(0,r)\setminus \{0\}.$ But $\xi_X$ is bounded, so the singularity on $0$ is removable and $\xi_X$ is holomorphic on $D:= D(0,r).$ 
 
 As a consequence, we can define $\tilde{\Psi} : \biconm{C_\Fi} \rightarrow H^\infty(D)$ as $\tilde{\Psi}(X)= \xi_X$. In particular, it follows that
 \begin{equation}\label{propiedad disco}
 Xe^{\lambda h} = \Psi(X)e^{\lambda h}= \tilde{\Psi}(X)(e^\lambda) e^{\lambda h} \qquad (\lambda \in G),
 \end{equation}
 
  and by construction $\tilde{\Psi}$ is an injective algebra homomorphism. 
 
 Now, to finish the proof, we are showing that $\Psi$ maps $\overline{\alg(C_\Fi)}^\sigma$ onto $H^\infty(D)$, which will yield that $\Psi$ is surjective and, by injectivity, $\overline{\alg(C_\Fi)}^\sigma=\biconm{C_\Fi},$ as we want to prove.
 
 Consider then $\xi \in H^\infty(D)$, we are constructing a sequence of polynomials $p_n$ such that $p_n(C_\Fi)$ converges to some operator $X$ in the weak operator topology satisfying $\tilde{\Psi}(X) = \xi.$ By \cite[Lemma 2.10]{LLPR2} there exists a sequence of polynomials $(p_n)_{n\in \N}$ that converges pointwisely to $\xi$ and satisfy that $\norm{p_n}_{H^\infty(D)} \leq \norm{\xi}_{H^\infty(D)}$ for every $n\in \N.$ Thus, consider $X_n = p_n(C_\Fi)$ for each $n\in \N.$ We have, by \eqref{hipotesis norma polinomios}, for every $n\in \N$,
 $$\norm{X_n}=\norm{p_n(C_\Fi)} \leq C \norm{p_n}_{H^\infty(D)} \leq C\norm{\xi}_{H^\infty(D)},   $$ so the sequence $(X_n)_{n\in \N}$ is norm bounded. Thus, since $1<p<\infty,$ there exists a subsequence $(X_{n_k})_{k\in \N}$ which converges in the weak operator topology to some $X\in \overline{\alg(C_\Fi)}^\sigma.$ We are showing that $\tilde{\Psi}(X) = \xi,$ which will finish the proof. For such a purpose, consider $\lambda \in G$ and $f \in H^q$ with $\frac{1}{p}+\frac{1}{q}=1$ and observe that
 \begin{equation*}
  \begin{split}
\pe{Xe^{\lambda h},f} = \pe{\lim\limits_k X_{n_k}e^{\lambda h},f} = \pe{\lim\limits_k \tilde{\Psi}(X_{n_k})(e^\lambda)e^{\lambda h},f} = \pe{\lim\limits_k p_{n_k}(e^\lambda)e^{\lambda h},f} = \pe{\xi(e^\lambda)e^{\lambda h},f},
  \end{split}
 \end{equation*}
so we deduce that $Xe^{\lambda h} = \xi(e^\lambda)$ for every $\lambda \in G$, so by \eqref{propiedad disco} $\tilde{\Psi}(X) = \xi,$ as we wanted to show. \hfill  $\Box$

\begin{remark}
It is also worth mentioning that not every composition operator belonging to a hyperbolic semigroup has the double commutant property, although its associated Koenigs domain has desirable geometric properties. For instance, if one considers $\Fi$ to be a hyperbolic automorphism of $\D$, it is always embedded in a hyperbolic group of automorphisms whose Koenigs domain is a horizontal strip. It was shown in \cite{LLPR2} that the induced composition operator acting on $H^2$ does not enjoy the double commutant property. 

	Indeed, in \cite[Theorem 5.6]{LLPR2} it is proved  that $\biconm{C_\Fi} = \overline{\alg(C_\Fi,C_\Fi^{-1})}^\sigma.$ The methods are quite similar to the ones presented in the proof of Theorem \ref{teorema doble conmutante hiperbólico}, with the main difference that the exponential maps the point spectrum of the associated infinitesimal generator onto an annulus instead of a punctured disc, which in some sense `forces' to consider $C_\Fi^{-1}$ in order to approximate with polynomials in the weak operator topology.
	 Following this line of thinking, consider now $(\Fi_t)_{t\geq 0}$ a hyperbolic semiflow such that $\Omega$ properly contains a horizontal strip. In particular, the induced composition operators $C_{\Fi_t} \in \EL(H^2)$ are not invertible.  This raises the question whether the equality $\biconm{C_\Fi}=\overline{\alg(C_\Fi)}^\sigma$ cannot hold in this setting.
\end{remark}
\subsection{Extended eigenvalues and strong compactness}\label{subseccion 1}
Given an operator $T\in \EL(\X)$, a complex number $\lambda \in \C$ is said to be an \textit{extended eigenvalue} for $T$ is there exists $A\in \EL(\X)$ such that $TA=\lambda AT.$ In such a case, we say that $A$ is an \textit{eigenoperator} for $T$ associated to the extended eigenvalue $\lambda$. The set of extended eigenvalues of $T$ is denoted as $\textnormal{Ext}(T)$. 

\medskip

Recall that, in the proof of Theorem \ref{teorema conmutante no abeliano hiperbolico}, we have shown the identity \eqref{eigenoperator}, which in particular shows that $M_{e^{ix h}}$ is an eigenoperator for $C_{\Fi_t}$ associated to the extended eigenvalue $e^{itx}.$ Exploiting such a property, we are able to state the following result, which describes the extended eigenvalues of composition operators embedded in hyperbolic semigroups: 
\begin{theorem}\label{extended eigenvalues hiperbolico}
	Let $\Fi : \D \rightarrow\D$ be a holomorphic function such  that there exists a hyperbolic holomorphic semiflow $(\Fi_t)_{t\geq 0}$ with $\Fi = \Fi_1.$ Let $\Omega$ be its Koenigs domain and assume that $\inte(\overline{\Omega}) = \Omega$ and that $\C\setminus \overline{\Omega}$ has at most two connected components. Then, 
	\begin{enumerate}
		\item [(a)] Assume that $\Omega$ does not contain any horizontal strip, and let $\rho = \inf \{\PR(z) : z\in \Omega \}.$
		\begin{enumerate}
		\item [(i)] If $\rho > - \infty$, then  $\textnormal{Ext}(C_\Fi) = \overline{\D}\setminus \{0\}.$
		\item [(ii)] If $\rho = - \infty,$ then $\T\subseteq \textnormal{Ext}(C_\Fi)\subseteq \overline{\D}\setminus \{0\}.$
		\end{enumerate}
		\item [(b)] Assume that $\Omega$ contains a horizontal strip. Then $\textnormal{Ext}(C_\Fi) = \T.$
	\end{enumerate}
\end{theorem}
\begin{proof}
First, let us assume that $\Omega$ does not contain any horizontal strip. By Theorem \ref{teorema betsakos} and \eqref{spectral mapping theorem point spectrum}, together with the fact that $C_\Fi$ is injective, it follows that $$D(0,e^{\pi/p \gamma })\setminus \{0\}\subseteq \sigma_p(C_\Fi\mid_{H^p})  \subseteq \overline{D(0,e^{\pi/p \gamma })}\setminus \{0\}.$$
Now, assume that $\rho >-\infty.$ In particular, it follows that $e^{\lambda h} \in H^\infty$ for every $\lambda \in \C$ such that $\PR(\lambda) < 0.$ Moreover, $C_{\Fi}e^{\lambda h} = e^{\lambda}e^{\lambda h}$. By considering the multiplication operator $M_{e^{\lambda h}}$ and imitating the argument used to show \eqref{eigenoperator}, we deduce that $$M_{e^{\lambda h}}C_\Fi = e^{\lambda}C_\Fi M_{e^{\lambda h}}$$ for every $\lambda$ with $\PR(\lambda < 0).$ This shows that $ \overline{\D}\setminus \{0\}\subseteq \textnormal{Ext}(C_\Fi).$ For the reverse inclusion, first observe that $0\notin \textnormal{Ext}(C_\Fi)$ since $C_\Fi$ is injective. Thus, it is enough to recall that $C_\Fi$ has rich point spectrum by Lemma \ref{lema rich point spectrum} and apply \cite[Theorem 3.1]{LLPZ} to deduce that  every $\lambda \in \textnormal{Ext}(C_\Fi)$ has modulus $|\lambda|\leq 1$, so $\textnormal{Ext}(C_\Fi)= \overline{\D}\setminus\{0\},$ as we wanted. 

Suppose now that $\rho = -\infty.$ In this case, we have that $e^{\lambda h} \in H^\infty$ if and only if $\PR(\lambda) = 0.$ Then, by \eqref{eigenoperator} we deduce that $\T\subseteq \textnormal{Ext}(C_\Fi).$ On the other hand, by recalling that $0\notin \textnormal{Ext}(C_\Fi)$ and again using that $C_\Fi$ has rich point spectrum, we deduce that $\textnormal{Ext}(C_\Fi)\subseteq \overline{\D}\setminus \{0\}.$

\medskip

To finish the proof, assume that $\Omega$ contains a horizontal strip. In such a case, again by Theorem \ref{teorema betsakos} and \eqref{spectral mapping theorem point spectrum} we have 
$$\{ \lambda \in \C: e^{- \frac{\pi}{p\beta_{\textnormal{max}}}} <|\lambda|<e^{\frac{\pi}{p\gamma}} \} \subseteq \sigma_p(C_\Fi \mid_{H^p})\subseteq \{ \lambda \in \C: e^{- \frac{\pi}{p\beta_{\textnormal{max}}}} <|\lambda| \leq e^{\frac{\pi}{p\gamma}} \}. $$
By \eqref{eigenoperator} we deduce that $\T \subseteq \textnormal{Ext}(C_\Fi).$ Finally, it is enough to recall again that $C_\Fi$ has rich point spectrum and apply \cite[Theorem 3.3]{LLPZ} to deduce that $\textnormal{Ext}(C_\Fi)=\T$, as we wanted to show.
\end{proof}

We end this section with a consequence regarding the strong compactness of composition operators embedded in $C_0$-semigroups. An algebra $\mathcal{A}\subset \EL(\X)$ is said to be \textit{strongly compact} if every bounded subset of $\A$ is relatively compact in the strong operator topology. In particular, an operator $T\in \EL(\X)$ is strongly compact if $\alg(T)$ is a strongly compact algebra.

\begin{theorem}\label{strong compactness hiperbolico}
	Let $\Fi : \D \rightarrow\D$ be a holomorphic function such  that there exists a hyperbolic semiflow $(\Fi_t)_{t\geq 0}$ with $\Fi = \Fi_1.$ Let $\Omega$ be its Koenigs domain  and assume that $\inte(\overline{\Omega}) = \Omega$ and that $\C\setminus \overline{\Omega}$ has at most two connected components. Then, $C_\Fi$ is strongly compact.
\end{theorem}
\begin{proof}
This result follows as a byproduct of Proposition \ref{proposicion densidad} and \cite[Theorem 2.1]{FL}.
\end{proof}
\section{Parabolic semiflows}\label{seccion parabolicos}
In this section, we consider composition operators $C_\Fi$ induced by symbols $\Fi$ that can be embedded in a parabolic semiflow. As we noted previously in Section \ref{seccion introduccion}, we will distinguish between zero hyperbolic step and positive hyperbolic step semiflows.

\medskip

The kind of hyperbolic step determines some geometric aspects of the Koenigs domain of the semiflow. The following result holds, as a byproduct of \cite[Theorem 9.3.5 and Proposition 9.3.10]{BCD}:

\begin{proposition}\label{proposicion geometria}
Let $(\Fi_t)_{t\geq 0}$ be a parabolic semiflow, and $\Omega$ be its Koenigs domain. Then:
\begin{enumerate}
		\item [(i)] If $(\Fi_t)_{t\geq 0}$ is of positive hyperbolic step, then $$  \bigcup\limits_{t\geq 0} (\Omega-t) \in \{ \{w\in \C:  \Impart(w)>\nu \}, \{w\in \C:  \textnormal{Im}(w)<\nu \}\} $$ for some $\eta \in \R.$
	\item [(ii)] If $(\Fi_t)_{t\geq 0}$ is of zero hyperbolic step, then $$ \bigcup\limits_{t\geq 0} (\Omega-t) = \C.$$ In particular, $\Omega$ is not contained in any horizontal half-plane.
\end{enumerate}
\end{proposition}
Unlike the hyperbolic setting, the point spectrum of the infinitesimal generator for parabolic semiflows is not totally understood yet. In order to illustrate some known situations, we recollect some results from Betsakos that can be found in \cite[Section 6]{Betsakos}.
\begin{theorem}\label{betsakos parabolico}\cite{Betsakos}
	Let $(\Fi_t)_{t\geq 0}$ be a parabolic semiflow, and let $\Omega$ be its Koenigs domain. Let $(C_{\Fi_t})_{t\geq 0}\subset \EL(H^p)$ $(1\leq p < \infty)$ be the corresponding $C_0$-semigroup of composition operators, and let $\Delta$ be its infinitesimal generator.
	\begin{enumerate}
		\item [(i)] $\displaystyle{ \sigma_p(\Delta\mid_{H^p})\subseteq \{ \lambda \in \C: \PR(\lambda)\leq 0 \}}.$
		\item [(ii)] Assume that $\Omega$ is contained in a horizontal half-plane of the form $\{ w\in \C: \Impart(w) > -\eta \},$ $\eta >0$, and that $\partial \Omega$ is contained in a horizontal strip. Then,
		$$\sigma_p(\Delta\mid_{H^p}) = \{ \lambda \in \C: \PR(\lambda) = 0, \Impart(\lambda)\geq 0 \}.$$
		\item [(iii)] Assume that $\C\setminus \Omega$ is contained in a horizontal strip. Then,
		$$\sigma_p(\Delta\mid_{H^p}) = \{0\}.$$
	\end{enumerate}
\end{theorem}
Observe that the spectral picture described in $(iii)$ can only happen if the semiflow is of zero hyperbolic step, due to Proposition \ref{proposicion geometria}.

\medskip

Now, we state a result concerning the connectedness of $\sigma_p(\Delta)$, which will be crucial in order to state the functional (bi)commutant property for parabolic semiflows:

\begin{lemma}\label{lema conexion}
	Let $(\Fi)_{t\geq 0}$ be a parabolic semiflow, let $\mathcal{T}=(C_{\Fi_t})_{t\geq 0}$ be the corresponding $C_0$-semigroup of composition operators acting on $H^p$ for $1\leq p < \infty$, and let $\Delta$ be its infinitesimal generator. Then, $\sigma_p(\Delta)$ is connected.
\end{lemma}
\begin{proof}
	First, recall that, by \eqref{expresion generador infinitesimal}, $\Delta \mathbbm{1} = 0,$ so $0 \in \sigma_p(\Delta).$ We are showing that $\sigma_p(\Delta)$ is star-shaped at $0$, that is, the segment $[0,\lambda]\subset \sigma_p(\Delta)$ for every $\lambda \in \sigma_p(\Delta),$ which implies the connectedness of $\sigma_p(\Delta).$ Let $\lambda \in \sigma_p(\Delta)$ and $t\in (0,1)$. Recall that $t\lambda \in \sigma_p(\Delta)$ if and only if $e^{t\lambda h} \in H^p$. We have, for each $z\in \D$,
	$| e^{t\lambda h}|^p = e^{pt\PR(\lambda h)}.$ Now, if $z\in \D$ is such that $\PR(\lambda h(z)) \leq 0$, then $e^{pt\PR(\lambda h(z))} \leq 1,$ and if $z \in \D$ is such that $\PR(\lambda h(z))> 0$, then $e^{pt\PR(\lambda h(z))}\leq e^{p\PR(\lambda h(z))} = \left| e^{\lambda h(z)}\right|^p.$ Hence, $|e^{t\lambda h(z)}|^p\leq \max\{1, |e^{\lambda h(z)}|^p \}$ for every $z\in \D$, and since $e^{\lambda h}\in H^p$, we deduce that $e^{t\lambda h} \in H^p$ as well, and the proof is done.
\end{proof}
Regarding the density of the associated eigenfunctions, the next result holds, following the lines of Proposition \ref{proposicion densidad}:
\begin{proposition}\label{densidad parabolico}\cite[Theorems 1.2 and 1.3]{BGY}
	Let $(\Fi_t)_{t\geq 0}$ be a parabolic semiflow, and let  $\Omega$ be its associated Koenigs domain. Assume that either
	\begin{enumerate}
		\item [(i)] $(\Fi_t)_{t\geq 0}$ is of positive hyperbolic step, $\inte(\overline{\Omega})=\Omega$ and $\C\setminus\overline{\Omega}$ is connected; or
		\item [(ii)]  $(\Fi_t)_{t\geq 0}$ is of zero hyperbolic step, $\Omega$ is contained in a half-plane and $\inte(\overline{\Omega})= \Omega.$
	\end{enumerate}
Then, if $h$ is the Koenigs function associated to $(\Fi_t)_{t\geq 0},$ the linear manifold
$$\Span\{e^{\lambda h} : \lambda \in \sigma_p(\Delta\mid_{H^p}) \}$$ is dense on $H^p$ for  $1\leq p <\infty.$
\end{proposition}

\subsection{The functional (bi)commutant property} 
With the previous results in hand, we can state our main result for the functional bicommutant property for this setting:

\begin{theorem}\label{teorema parabolico}
	Let $\Fi : \D \rightarrow\D$ be a holomorphic function such  that there exists a parabolic semiflow $(\Fi_t)_{t\geq 0}$ with $\Fi = \Fi_1.$ Let $\Omega$ be its Koenigs domain. Assume that either 
	\begin{enumerate}
		\item [(i)] $(\Fi_t)_{t\geq 0}$ is of positive hyperbolic step, $\inte(\overline{\Omega})=\Omega$ and $\C\setminus\overline{\Omega}$ is connected; or
		\item [(ii)]  $(\Fi_t)_{t\geq 0}$ is of zero hyperbolic step, $\Omega$ is contained in a half-plane and $\inte(\overline{\Omega})= \Omega.$
	\end{enumerate}
Then,  $C_\Fi \in \EL(H^p)$ has the functional bicommutant property for $1\leq p<\infty$. In particular, there exist no non-trivial idempotents lying in $\biconm{C_\Fi}.$
\end{theorem}
\begin{proof}
Let $G = \sigma_p(\Delta\mid_{H^p})$, and recall that it is a connected set by Lemma \ref{lema conexion}. We will construct an algebra isomorphism from $\biconm{C_\Fi}$ to a certain subalgebra of $\mathcal{C}(G).$ For this purpose, let $X \in \biconm{C_\Fi}$ and observe that, for every $\lambda \in G,$ \eqref{igualdad biconmutante} holds. As in the proof of Theorem \ref{teorema hiperbolico}, it follows that there exists a map $\nu_X : G \rightarrow \C$ such that $Xe^{\lambda h} = \nu_X(\lambda)e^{\lambda h}.$ We are showing that $\nu_X$ is continuous.  Observe that the map $\lambda \in G \mapsto e^{\lambda h} \in H^p$ is continuous considering the weak topology of $H^p.$ Thus, arguing as in the proof of Theorem \ref{teorema hiperbolico}, we deduce that $\nu_X(\lambda)=e^{-\lambda h(0)}\pe{Xe^{\lambda h},\mathbbm{1}},$ which yields the continuity of $\Psi.$ Thus, $\nu_X \in \mathcal{C}(G)$ for every $X\in \biconm{C_\Fi}.$

At this point, we may define the map $\Psi : \biconm{C_\Fi} \rightarrow \mathcal{C}(G),$ which satisfies the identity \eqref{igualdad clave}. Hence, it is enough to mimic the arguments of the proof of Theorem \ref{teorema hiperbolico}, having into account that the linear manifold $$\Span \{e^{\lambda h} : \lambda \in G \}$$ is dense on $H^p$ by Proposition \ref{densidad parabolico}, to deduce that $\Psi$ is an injective algebra homomorphism. Finally,  consider $\Phi := \Psi^{-1} : \ran(\Psi) \rightarrow \biconm{C_\Fi},$ which is an algebra isomorphism. Then, $C_\Fi$ has the functional bicommutant property and the proof is done.
\end{proof}
As in the hyperbolic case (see Corollary \ref{corolario conmutante semigrupo hiperbolico}), the previous proof allows us to derive the following result, regarding the commutant of the $C_0-$semigroup:

\begin{corollary}\label{corolario conmutante semigrupo parabolico}
Let $(\Fi_t)_{t\geq0}$ be a parabolic flow, let Let $\Omega$ be its Koenigs domain and denote by $\mathcal{T}= (C_{\Fi_t})_{t\geq 0}$ it associated $C_0-$semigroup acting on $H^p$ for $1\leq p <\infty$.  Assume that either 
\begin{enumerate}
	\item [(i)] $(\Fi_t)_{t\geq 0}$ is of positive hyperbolic step, $\inte(\overline{\Omega})=\Omega$ and $\C\setminus\overline{\Omega}$ is connected; or
	\item [(ii)]  $(\Fi_t)_{t\geq 0}$ is of zero hyperbolic step, $\Omega$ is contained in a half-plane and $\inte(\overline{\Omega})= \Omega.$
\end{enumerate}
  Then, $\mathcal{T}'$ is abelian and does not contain any non-trivial idempotent.
\end{corollary}

 Unlike the case of the hyperbolic semiflows (see Theorem \ref{teorema conmutante no abeliano hiperbolico}) there are situations in which one can derive the functional commutant property for composition operators embedded in parabolic semiflows. In this respect, we introduce the following lemma, which is a straightforward adaptation of an argument exposed by Shapiro for linear fractional parabolic composition operators \cite[p. 862]{Shapiro}:
 
 \begin{lemma}\label{lema dimension kernel}
Let $\Fi : \D\rightarrow \D$ be a holomorphic function, and assume that
\begin{equation}\label{blaschke condition}
\sumn 1-|\Fi_n(0)| = \infty,
\end{equation} where $\Fi_n = \underbrace{\Fi \circ \ldots \circ \Fi}_{n \text{ times}}$ for every $n\in \N.$ Then, if $\lambda \in \sigma_p(C_\Fi),$ $\ker (C_\Fi-\lambda I)$ is one-dimensional.
 \end{lemma}
It is necessary to remark that, if \eqref{blaschke condition} holds for a holomorphic map $\Fi:\D\rightarrow\D$ embedded in a parabolic semiflow $(\Fi_t)_{t\geq 0}$, then it is of zero hyperbolic step \cite[Lemma 2.6]{BMS}, and its associated Koenigs domain is not contained in any horizontal half-plane.
 \begin{theorem}\label{functional commutant parabolico}
Let $\Fi : \D \rightarrow\D$ be a holomorphic function such  that there exists a parabolic semiflow $(\Fi_t)_{t\geq 0}$ with $\Fi = \Fi_1,$ and let $\Omega$ be its Koenigs domain. Assume that it is contained in a half-plane and $\inte(\overline{\Omega})=\Omega$. If 
$$\sumn 1-|\Fi_n(0)| = \infty,$$
then $C_\Fi \in \EL(H^p)$ has the functional commutant property for $1\leq p<\infty$. In particular, there exist no non-trivial idempotents lying in $\conm{C_\Fi}.$
 \end{theorem}
 \begin{proof}
Since $\sumn 1-|\Fi_n(0)| = \infty$, it follows that $(\Fi_t)_{t\geq 0}$ is of zero hyperbolic step and the eigenfunctions $e^{\lambda h}$ has multiplicity one for $C_\Fi$ for all $\lambda \in \sigma_p(\Delta\mid_{H^p})$ by Lemma \ref{lema dimension kernel}. Then, if $X\in\conm{C_\Fi}$, observe that
$$Xe^{\lambda h} = Xe^{-\lambda}C_\Fi e^{\lambda h} = e^{-\lambda} C_\Fi(Xe^{\lambda h}),$$ so we deduce that there exists a map $\nu_X: \sigma_p(\Delta\mid_{H^p}) \rightarrow \C$ such that $Xe^{\lambda h} = \nu_X Xe^{\lambda h}.$ From now, it is enough to argue as in the proof of Theorem \ref{teorema parabolico} to deduce the result.
 \end{proof}

Let us deal now with another interesting consequence.  The next result extendes the Montes-Rodríguez, Ponce-Escudero and Shkarin result for non-automorphic parabolic linear fractional maps discussed in Section \ref{seccion introduccion}:

\begin{corollary}\label{corolario parabolico no-automorfismo}
	Let $\Fi : \D\rightarrow \D$ be a linear fractional non-automorphic parabolic map. Then, $C_\Fi \in \EL(H^p)$ does not commute with any non-trivial idempotent for all $1\leq p < \infty.$
\end{corollary}
\begin{proof}
	 By conjugation, we may assume that
	
	$$\Fi_a(z) = \frac{(2-a)z+a}{-az+2+a},$$
	
	where $\PR(a)>0$.  This function $\Fi_a$ can be easily shown to be embedded in a parabolic semiflow of the form $(\Fi_{at}(z))_{t\geq 0}$, whose Koenigs function is given by $h(z) = \frac{1}{a} \frac{1+z}{1-z}.$ As a consequence, the Koenigs domain $\Omega$ is a non-horizontal half-plane, so the hypotheses of Theorem \ref{functional commutant parabolico} on $\Omega$ are verified. 
	
	Then, if we show that
	$$\sumn 1-|\Fi_n(0)| = \infty,$$ we can apply Theorem \ref{functional commutant parabolico} and the proof will be done. We have that $\Fi_n(0) = \frac{na}{2+na}$ for every $n\in \N,$ so $|\Fi_n(0)| = \frac{n|a|}{\sqrt{4+4n\PR(a)+n^2|a|^2}}.$ Hence, by considering $1-|\Fi_n(0)|$ and multiplying by the conjugate of the numerator we obtain  $$1-|\Fi_n(0)| = \frac{4+4n\PR(a)}{4+4n\PR(a)+n^2|a|^2+n|a|\sqrt{4+4n\PR(a)+n^2|a|^2}}.$$ At this point, it is enough to apply the limit comparion test with the sequence $1/n$ to deduce that $\sumn 1-|\Fi_n(0)|$ diverges, as we wished. This ends the proof.
\end{proof}

On the other hand, there are some geometric configurations of the Koenigs domain of parabolic semigroups that lead to non-abelian commutants, which prevent $C_\Fi$ to have the functional commutant property or minimal commutant.

\begin{theorem}\label{conmutante abeliano parabolico}
Let $\Fi : \D \rightarrow \D$ be a holomorphic function such that there exists a parabolic semiflow $(\Fi_t)_{t\geq 0}$ with $\Fi = \Fi_1.$ Let $\Omega$ be its Koenigs domain and assume it is contained in a horizontal half-plane of the form $\{w\in\C: \Impart(w)> -\eta\}, \eta >0$. Then, $\conm{C_\Fi}$ is not abelian. In particular, $C_\Fi$ does not have the functional commutant property or minimal commutant.
\end{theorem}
\begin{proof}
Let $h$ be the Koenigs function associated to $(\Fi_t)_{t\geq 0}.$ Observe that, by the hypothesis on $\Omega,$ the functions $e^{ix h}$ belongs to $H^\infty$ for every $x\geq 0.$   Then, one may consider the multiplication operator $M_{e^{ix h}}.$ Repeating the arguments of the proof of Theorem  \ref{teorema conmutante no abeliano hiperbolico}, we deduce that \eqref{eigenoperator} also holds for this situation, and the same ideas exposed in the aforementioned proof yields that $\conm{C_\Fi}$ is not abelian.
\end{proof}
The previous result can be applied to show, for instance, that if $\Fi$ is a parabolic automorphism of $\D$ then $C_\Fi \in \EL(H^p)$ does not have abelian commutant for any $1\leq p < \infty.$ Since it is an automorphism, it can be embedded in a parabolic group whose Koenigs domain is an upper or a lower half-plane, so the conditions of Theorem \ref{conmutante abeliano parabolico} are satisfied.

\subsection{Extended eigenvalues and strong compactness}\label{subseccion 2}
We finish this section with two results about the extended eigenvalues and the strong compactness of composition operators whose symbol is embedded into a parabolic semiflow.

\medskip

With respect to extended eigenvalues, one might consider different cases regarding the geometry of the point spectrum of the infinitesimal generator. We prefer to present a less general but simpler result, which will totally characterize the extended eigenvalues for some geometrical pictures of the Koenigs domain:

\begin{theorem}\label{extended eigenvalues parabolico}
	Let $\Fi : \D \rightarrow \D$ be a holomorphic function such that there exists a parabolic semiflow $(\Fi_t)_{t\geq 0}$ with $\Fi = \Fi_1.$ Let $\Omega$ be the Koenigs domain and assume it is contained in a horizontal half-plane of the form $\{w\in\C: \Impart(w)> -\eta\}, \eta >0$ and that $\partial \Omega$ is contained in a horizontal strip. Assume further $\inte(\overline{\Omega})=\Omega$ and $\C\setminus\overline{\Omega}$ is connected. Then, $$\textnormal{Ext}(C_\Fi) = \T.$$
\end{theorem}
\begin{proof}
	As it was shown in the proof of Theorem \ref{conmutante abeliano parabolico}, the identity \eqref{eigenoperator} holds for $C_{\Fi_t}$ ($t\geq 0$) and $x\geq 0.$ Moreover, using the same arguments that in the proof of Theorem \ref{extended eigenvalues hiperbolico}, we deduce that $\T \subseteq \textnormal{Ext}(C_\Fi).$ Finally, by Theorem \ref{betsakos parabolico} and \eqref{spectral mapping theorem point spectrum} we deduce that $\sigma_p(C_\Fi) = \T$, and applying Proposition \ref{densidad parabolico} together with \cite[Lemma 2.6]{LLPR3} it follows that $\textnormal{Ext}(C_\Fi)\subseteq \T$, and the proof is done.
\end{proof}

Finally, we deal with the strong compactness:

\begin{theorem}\label{strong compactness parabolico}
	Let $\Fi : \D \rightarrow\D$ be a holomorphic function such  that there exists a parabolic semiflow $(\Fi_t)_{t\geq 0}$ with $\Fi = \Fi_1.$ Let $\Omega$ be its Koenigs domain. Assume that either 
\begin{enumerate}
	\item [(i)] $(\Fi_t)_{t\geq 0}$ is of positive hyperbolic step, $\inte(\overline{\Omega})=\Omega$ and $\C\setminus\overline{\Omega}$ is connected; or
	\item [(ii)]  $(\Fi_t)_{t\geq 0}$ is of zero hyperbolic step, $\Omega$ is contained in a half-plane and $\inte(\overline{\Omega})= \Omega.$
\end{enumerate}
Then, $C_\Fi$ is strongly compact. Moreover, if $$\sumn 1-|\Fi_n(0)|= \infty,$$ then $\conm{C_\Fi}$ is a strongly compact algebra.
\end{theorem}
\begin{proof}
	The strong compactness of $C_\Fi$ follows as a byproduct of \cite[Theorem 2.1]{FL}  and Proposition \ref{densidad parabolico}. Moreover, the second part of  of \cite[Theorem 2.1]{FL} together with Lemma \ref{lema dimension kernel} yield the strong compactness of $\conm{C_\Fi},$ and the proof is done.
\end{proof}
\section{Elliptic semiflows}\label{seccion elipticos}
In this section, we deal with composition operators $C_\Fi$ such that its symbol $\Fi$ can be embedded in an elliptic semiflow. Recall that a holomorphic semiflow $(\Fi_t)_{t\geq 0}$ is elliptic if there exists some $t_0 > 0$ such that $\Fi_{t_0}$ has a fixed point in $\D$. Indeed, one can see (for instance, \cite[Remark 8.3.4]{BCD}) that $\Fi_t$ has a fixed point in $\D$ for every $t\geq 0.$ Therefore, our results will focus on the more general setting of holomorphic self-maps of $\D$ with fixed points in the unit disc, instead of considering just symbols $\Fi$ that may be embedded in an elliptic semigroup.

\medskip

It is not difficult to generate non-trivial idempotents (or orthogonal projections in the case $p=2$) lying in the bicommutant of $C_\Fi$ whenever $\Fi$ has a fixed point in $\D$, which, in particular, implies that $C_\Fi$ does not have the functional commutant (or bicommutant) property. These kinds of results were first stated by Guyker \cite{Guyker} for $p=2.$ We include here a proof of the existence of such idempotents for every $1< p < \infty$, for the sake of completeness of the manuscript.

 To prove such a result, let us recall that, given $f \in H^p$ and $g \in H^q$ where $1<p<\infty$ and $\frac{1}{p}+\frac{1}{q} = 1$, the operator $f\otimes g \in \EL(H^p)$ is a rank-one operator that acts as $$(f\otimes g)h = \pe{g,h}f.$$ It is straightforward to check that the operator $\mathbbm{1}\otimes\mathbbm{1}$ is an idempotent, and in particular an orthogonal projection when $p=2.$

\begin{theorem}\label{teorema eliptico}
	Let $\Fi :\D\rightarrow \D$ be a holomorphic function with a fixed point in $\D$ and consider $C_\Fi \in \EL(H^p)$, $1< p < \infty.$ Then, there exist non-trivial idempotents lying in $\biconm{C_\Fi}.$ As a matter of fact, $C_\Fi$ does not have the functional commutant or bicommutant property.
\end{theorem}
\begin{proof}
	By conjugation, we may assume without loss of generality that the fixed point of $\Fi$ in $\D$ is $0$. Recall that, if $\Fi$ is an automorphism, then $\Fi(z) = \omega z,$ with $\omega\in \T.$ We may also assume without loss of generality that $\textnormal{Arg}(\omega)$ is not a rational multiple of $\pi$, since otherwise $C_\Fi$ is scalar and has plenty of non-trivial idempotents in its bicommutant \cite[Theorem 1.2]{Smith}.

	Thus, let us show that $\mathbbm{1}\otimes\mathbbm{1}$ lies in the bicommutant of $C_\Fi.$ Let $A\in \conm{C_\Fi}$. We have $$A\mathbbm{1}\otimes\mathbbm{1} = (A\mathbbm{1})\otimes \mathbbm{1} \qquad \mathbbm{1}\otimes\mathbbm{1}A = \mathbbm{1}\otimes(A^*\mathbbm{1}).$$
We are showing that $A\mathbbm{1} = \lambda \mathbbm{1}$  and $A^*\mathbbm{1} = \overline{\lambda}\mathbbm{1}$ for some $\lambda \in \C$, which will finish the proof. Recall that $C_\Fi \mathbbm{1} = \mathbbm{1}$ and $C_\Fi^*\mathbbm{1} = \mathbbm{1},$ and since $A\in \conm{C_\Fi}$ and $A^*\in \conm{C_\Fi^*}$ it follows that $A\mathbbm{1}\in \ker(C_\Fi-I)$ and $A^*\mathbbm{1}\in \ker(C_\Fi^*-I).$ Thus, if we show that the eigenvalue $1$ has multiplicity one for both $C_\Fi$ and $C_\Fi^*$ we will be done. It will follow from the fact that if $A\mathbbm{1} = \lambda \mathbbm{1}$ and $A^*\mathbbm{1} = \mu \mathbbm{1},$ we would have
$$\lambda = \pe{A\mathbbm{1},\mathbbm{1}} = \pe{\mathbbm{1},A^*\mathbbm{1}} = \pe{\mathbbm{1},\mu\mathbbm{1}} = \overline{\mu},$$ so $\mu = \overline{\lambda},$ as we wanted. We distinguish two cases:
	
	\medskip
	
	\noindent \textit{Case 1:} $\Fi$ is an automorphism. Then, there exists $\omega \in \T$ such that $\Fi(z) = \omega z,$ where we may assume $\textnormal{Arg}(\omega)$ is not a rational multiple of $\pi.$ In addition, $C_\Fi^*f = f(\overline{\omega}z).$ If $C_\Fi f = f$ or $C_\Fi^*f = f,$ an easy computation involving Taylor series yields that $f$ is constant, so this case is proved. 
	
	\medskip
	
	\noindent \textit{Case 2:} $\Fi$ is not an automorphism. Let us denote as $\Fi_n = \underbrace{\Fi \circ \ldots \circ \Fi}_{n \text{ times}} $ for every $n\in \N.$ 	Now, by \cite[Proposition 1.8.3]{BCD} it follows that $\Fi_n \rightarrow 0$ in compact subsets of $\D.$ If $C_\Fi f = f,$ then $$f(\Fi_n(z)) = (C_{\Fi_n}f)(z) = (C_\Fi^n f)(z) = f(z)$$ for every $z\in \D$ and $n\in \N.$ As a consequence, $f(z) = f(0)$ for every $z\in \D$, so $f$ is constant, as we wished. Now, let $f\in H^q$ with $1/p+1/q = 1$. Assume $C_\Fi^*f =f,$ which implies $C_{\Fi_n}^*f=f$ for every $n\in \N.$ Then, for every $g\in H^p$ we have
	$$\pe{f,g} = \pe{C_{\Fi_n}^*f,g} = \pe{f, C_{\Fi_n}g}.$$ Now,  $C_{\Fi_n}g$ converges uniformly on compact subsets to $g(0)\mathbbm{1}$ and $\norm{C_{\Fi_n}g}= \norm{C_\Fi^ng} \leq \norm{g}$ since $\Fi_n(0) = 0$. These two facts, along with \cite[Chapter 20, Proposition 3.1.5]{Conway}, imply that $C_{\Fi_n}g$ converges to $g(0)\mathbbm{1}$ in the weak topology, so $\pe{f,C_{\Fi_n}g} \rightarrow \overline{g(0)}f(0),$ which implies that $f$ is constant, and the proof is done.
\end{proof}
At this point, it is worth noticing that in \cite[Theorem 6.1]{LLPR} it was shown that for a certain family of univalent symbols $\Fi$ fixing a point in $\D$, the induced composition operators $C_\Fi \in \EL(H^2)$ have minimal commutants, so in particular they have the double commutant property. Nevertheless, as it is shown in Theorem \ref{teorema eliptico}, these operators have non-trivial idempotents (indeed, non-trivial orthogonal projections) lying in $\biconm{C_\Fi},$ which illustrates that the concepts of minimal commutant or the double commutant property are not sufficient to assure the non-existence of non-trivial idempotents in the commutant or the bicommutant of an operator. This fact supports the necessity of introducing the definition of the functional commutant and the functional bicommutant property in order to assure the non-existence of non-trivial idempotents lying in the commutant or bicommutant of bounded operators.

\section*{Acknowledgments}
The author would like to express his gratitude towards Eva A. Gallardo-Gutiérrez, for her useful inputs that helped to improve the presentation and readability of this paper.

\section*{Conflict of interest statement}
The author states that there is no conflict of interest.

\section*{Data availability}
No data was used for the research described in the article.

\end{document}